\documentclass[a4paper, 12pt]{article} 
\usepackage{amssymb,latexsym} \author{Johannes Sj\"ostrand\\
\small Centre de Math\'ematiques Laurent Schwartz\\
\small Ecole Polytechnique\\
\small FR-91128 Palaiseau C\'edex\\
\footnotesize johannes@math.polytechnique\\
\footnotesize  and UMR 7640, CNRS
} \date{}
\title{Eigenvalue distribution for non-self-adjoint operators on
  compact manifolds with
  small multiplicative random perturbations}

\newtheorem{dref}{Definition}[section] 
\newtheorem{theo}[dref]{Theorem} \newtheorem{prop}[dref]{Proposition}
\newtheorem{remark}[dref]{Remark}

\newenvironment{proof}{\par\noindent{{\bf Proof.}}}{\hfill$\Box$
\medskip} 
 
 \newcommand\R{\mathbb{R}}

\newcommand{\ekv}[2]{\begin{equation}\label{#1}#2\end{equation}}
\newcommand{\eekv}[3]{\begin{eqnarray}\label{#1}#2 \\ #3
\nonumber\end{eqnarray}}
\newcommand{\eeekv}[4]{\begin{eqnarray}\label{#1}#2 \\ #3
\nonumber\\#4\nonumber\end{eqnarray}}

  \newcommand\iint{\int\hskip -2mm\int}

\newcommand{\no}[1]{(\ref{#1})} 
\begin{document}

\maketitle
\begin{abstract} In this work we extend a previous work about the Weyl
asymptotics of the distribution of eigenvalues of non-self-adjoint
diff\-erential operators with small multiplicative random
perturbations, by treating the case of operators on compact manifolds

\medskip
\par \centerline{\bf R\'esum\'e} Dans ce travail nous \'etendons 
un travail pr\'ec\'edent sur l'asymp\-totique de Weyl de la
distribution des valeurs propres d'op\'erateurs
diff\'erentiels avec des perturbations multiplicatives 
al\'eatoires petites,
en traitant le cas des op\'erateurs sur des vari\'et\'es compactes.

\end{abstract}

\tableofcontents

\section{Introduction}\label{int}
\setcounter{equation}{0}
This work is a direct continuation of \cite{Sj1}, devoted to
semi-classical pseudo\-differential operators on ${\bf R}^n$ with small multiplicative
random perturbations, which was partly based on the work by M.~Hager
and the author \cite{HaSj}. The main goal in the present work is to
obtain the same results as in \cite{Sj1} but with ${\bf R}^n$ replaced
by a compact smooth $n$-dimensional manifold $X$. Hopefully this
extension will make it possible to obtain almost sure Weyl asymptotics
for the large eigenvalues of elliptic operators on compact
manifolds. Such results in the case of $X=S^1$ have recently been
obtained by W. Bordeaux-Montrieux \cite{Bo}.

On $X$ we consider an $h$-differential operator $P$ which in local
coordinates takes the form,
\ekv{int.1}
{
P=\sum_{|\alpha |\le m}a_\alpha (x;h)(hD)^\alpha ,
} 
where we use standard multiindex notation and let
$D=D_x=\frac{1}{i}\frac{\partial }{\partial x}$. We assume that the 
coefficients $a_\alpha  $ are uniformly bounded in $C^\infty $ for
$h\in ]0,h_0]$, $0<h_0\ll 1$. (We will also discuss the case when we
only have some Sobolev space control of $a_0(x)$.) Assume
\eekv{int.2}
{
&&a_\alpha (x;h)=a_\alpha ^0(x)+{\cal O}(h) \mbox{ in }C^\infty ,}
{&&a_\alpha (x;h)=a_\alpha (x)\hbox{ is independent of }h\hbox{ for
  }|\alpha |=m.
}
Notice that this assumption is invariant under changes of local
coordinates. 

Also assume that $P$ is elliptic in the classical sense, uniformly
with respect to $h$:
\ekv{int.3}
{
|p_m(x,\xi )|\ge \frac{1}{C}|\xi |^m,
}
for some positive constant $C$, where
\ekv{int.4}
{
p_m(x,\xi )=\sum_{|\alpha |=m}a_\alpha (x)\xi ^\alpha 
}
is invariantly defined as a function on $T^*X$.
It follows that $p_m(T^*X)$ is a closed cone in ${\bf C}$ and we
assume that 
\ekv{int.5}
{
p_m(T^*X)\ne {\bf C}.
}
If $z_0\in {\bf C}\setminus p_m(T^*X)$, we see that $\lambda
z_0\not\in \Sigma (p)$ if $\lambda \ge 1$ is sufficiently large and
fixed, where $\Sigma (p):=p(T^*X)$ and $p$ is the semiclassical
principal symbol
\ekv{int.6}
{
p(x,\xi )=\sum_{|\alpha |\le m}a_\alpha ^0(x)\xi ^\alpha .
}
Actually, (\ref{int.5}) can be replaced by the weaker condition that
$\Sigma (p)\ne {\bf C}$.

\par Standard elliptic theory and analytic Fredholm theory now show
that if we consider $P$ as an unbounded operator: $L^2(X)\to L^2(X)$
with domain ${\cal D}(P)=H^m(X)$ (the Sobolev space of order $m$),
then $P$ has purely discrete spectrum.

\par We will need the symmetry assumption
\ekv{int'.7}{P^*=\Gamma P\Gamma ,}
where $P^*$ denotes the formal complex adjoint of $P$ in $L^2(X,dx)$,
with $dx$ denoting some fixed smooth postive density of integration
and $\Gamma $ is the antilinear operator of complex conjugation;
$\Gamma u=\overline{u}$. Notice that this assumption implies that 
\ekv{int'.8}
{
p(x,-\xi )=p(x,\xi ),
}
and conversely, if $p$ fulfills (\ref{int'.8}), then we get
(\ref{int'.7}) if we replace $P$ by $\frac{1}{2}(P+\Gamma P^*\Gamma $),
which has the same semi-classical principal symbol $p$.

\par Let $V_z(t):=\mathrm{vol\,}(\{ \rho \in {\bf R}^{2n};\,|p(\rho
)-z|^2 \le t\} )$. For $\kappa \in ]0,1]$, $z\in \Omega $, we consider
the property that
\ekv{int.6.2}{V_z(t)={\cal O}(t^ \kappa ),\ 0\le t \ll 1.} 
Since $r\mapsto p(x,r\xi )$ is a polynomial of degree $m$ in $r$ with
non-vanishing leading coefficient, we see that (\ref{int.6.2}) holds
with $\kappa =1/(2m)$.

The random potential will be of the form 
\ekv{int.6.3}
{q_\omega (x)=\sum_{0<\mu _k\le L}\alpha _k(\omega )\epsilon
_k(x),\ |\alpha |_{{\bf C}^D}\le R,}
where $\epsilon _k$ is the orthonormal basis of eigenfunctions of
$h^2\widetilde{R}$, where $\widetilde{R}$ is an $h$-independent
positive elliptic 2nd order operator on $X$ with smooth 
coefficients. Moreover, $h^2\widetilde{R}\epsilon _k=\mu
_k^2 \epsilon _k$, $\mu _k>0$. 
We choose $L=L(h)$, $R=R(h)$ in the interval
\eekv{int.6.4}
{h^{\frac{\kappa -3n}{s-\frac{n}{2}-\epsilon }}\ll L\le Ch^{-M},&&
M\ge \frac{3n-\kappa }{s-\frac{n}{2}-\epsilon },}
{\frac{1}{C}h^{-(\frac{n}{2}+\epsilon )M+\kappa -\frac{3n}{2}}\le R\le
 C h^{-\widetilde{M}},&& \widetilde{M}\ge \frac{3n}{2}-\kappa
  +(\frac{n}{2}+\epsilon )M,}
for some $\epsilon \in ]0,s-\frac{n}{2}[$, $s>\frac{n}{2}$,
so by Weyl's law for the large eigenvalues of elliptic
self-adjoint operators, the dimension $D$ is of the order of magnitude
$(L/h)^n$. We introduce the  small parameter 
$\delta =\tau _0 h^{N_1+n}$, $0<\tau _0\le \sqrt{h}$, where 
\ekv{int.6.4.3}
{
N_1:=\widetilde{M}+sM+\frac{n}{2}.
} 
The randomly perturbed operator is
\ekv{int.6.4.5}
{
P_\delta =P+\delta h^{N_1}q_\omega =:P+\delta Q_\omega .
}

\par The random variables $\alpha _j(\omega )$ will have a
joint probability distribution \ekv{int.6.5}{P(d\alpha )=C(h)e^{\Phi
(\alpha ;h)}L(d\alpha ),} where for some $N_4>0$,
\ekv{int.6.6}{ |\nabla _\alpha \Phi |={\cal
O}(h^{-N_4}),} and $L(d\alpha )$ is the
Lebesgue measure. ($C(h)$ is the normalizing constant, 
assuring that the probability of
$B_{{\bf C}^D}(0,R)$ is equal to 1.) 

\par We also need the parameter 
\ekv{int.6.7.5}{\epsilon _0(h)=(h^{\kappa }+h^n\ln 
\frac{1}{h})(\ln \frac{1}{\tau _0}+(\ln \frac{1}{h})^2)} and assume
that $\tau _0=\tau _0(h)$ is not too small, so that $\epsilon _0(h)$ is
small. Let $\Omega \Subset {\bf C}$ be open, simply connected not
entirely contained in $\Sigma (p)$. The main result of this work is:
\begin{theo}\label{int1} Under the assumptions above, let 
$\Gamma \Subset \Omega $ have smooth boundary, let $\kappa \in
]0,1]$ be the parameter in \no{int.6.3}, \no{int.6.4}, \no{int.6.7.5} and assume that 
\no{int.6.2} holds uniformly for $z$ in a
neighborhood of $\partial \Gamma $.  Then there
exists a constant $C>0$ such that for
$C^{-1}\ge r>0$,
$\widetilde{\epsilon }\ge C \epsilon _0(h)$ 
we have with probability 
\ekv{int.6.8}{
\ge 1-\frac{C\epsilon _0(h)}
{rh^{n+\max (n(M+1), N_4+\widetilde{M})}}
e^{-\frac{\widetilde{\epsilon }}{C\epsilon _0(h)}} }
that:
\eekv{int.7}
{
&&|
\#(\sigma (P_\delta )\cap \Gamma )-\frac{1}{(2\pi h)^n
}\mathrm{vol\,}(p^{-1}(\Gamma ))
|\le
}
{&&
\frac{C}{h^n}\left( \frac{\widetilde{\epsilon }}{r}
+C(r+\ln (\frac{1}{r})\mathrm{vol\,}(p^{-1}(\partial
\Gamma +D(0,r))))
 \right).}
Here $\#(\sigma (P_\delta )\cap \Gamma )$ denotes the number of
eigenvalues of $P_\delta $ in $\Gamma $, counted with their algebraic multiplicity.
\end{theo}

Actually, we shall prove the theorem for the slightly more general
operators, obtained by replacing $P$ by $P_0=P+\delta
_0(h^{\frac{n}{2}}q_1^0+q_2^0)$, where $\Vert q_1^0\Vert_{H^s_h}\le
1$, $\Vert q_2\Vert_{H^s}\le 1$, $0\le \delta _0\le h$. Here, $H^s$ is
the standard Sobolev space and $H_h^s$ is the same space with the
natural semiclassical $h$-dependent norm. See Section \ref{hs}.

As in \cite{HaSj} we also have a result valid simultaneously for a
family ${\cal C}$ of domains $\Gamma \subset \Omega $ satisfying the
assumptions of Theorem \ref{int1} uniformly in the natural sense:
With a probability 
\ekv{int.8}{
\ge 1-\frac{{\cal O}(1)\epsilon _0(h)}{r^2h^{n+\max (n(M+1), N_4+\widetilde{M})}}e^{-\frac{\widetilde{\epsilon }}{C\epsilon _0(h)}}, } the
estimate \no{int.7} holds simultaneously for all $\Gamma \in {\cal C}$.

\par In the introduction of \cite{Sj1} there is a discussion about the
choice of parameters and a corollary which carry over to the present
situation without any changes

\begin{remark}\label{int2}
{\rm When $\widetilde{R}$ has real coefficients, we may assume that the
eigenfunctions $\epsilon _j$ are real. Then (cf Remark 8.3 in \cite{Sj1}) we may
restrict $\alpha $ in (\ref{int.6.3}) to be in ${\bf R}^D$ so that
$q_\omega $ is real, still with $|\alpha |\le R$, and change
$C(h)$ in (\ref{int.6.5}) so that $P$ becomes a probability measure on 
$B_{{\bf R}^D}(0,R)$. Then Theorem \ref{int1} remains valid.}
\end{remark}
\begin{remark}\label{int3}
{\rm The assumption (\ref{int.6}) cannot be
completely eliminated. Indeed, let $P=hD_x+g(x)$ on ${\bf T}={\bf
  R}/(2\pi {\bf Z})$ where $g$ is smooth and complex valued. Then (cf
Hager \cite{Ha2}) the spectrum of $P$ is contained in the line 
$\Im z = \int_0^{2\pi }\Im g(x)dx/(2\pi )$. This line will vary only very
little under small multiplicative perturbations of $P$ so 
Theorem \ref{int1} cannot hold in this case.}
\end{remark}

The proof follows the general scheme of \cite{Sj1}, we will recall the
intermediate steps but give proofs only when there is a difference
between the case of ${\bf R}^n$ and that of compact
manifolds. Actually, there will also be some simplifications since we
have no support condition on the random potential.

\medskip
\par\noindent {\bf Acknowledgement.}
A large part of this work was completed while attending the
special program ``Complex Analysis of Several Variables'' at the 
Mittag-Leffler Institute in May--June 2008. We are grateful to the
organizers and the staff for very stimulating and pleasant working conditions.
\section{Semiclassical Sobolev spaces and multiplication}
\label{al}
\setcounter{equation}{0}
We let $H_h^s({\bf R}^n)\subset {\cal S}'({\bf R}^n)$, $s\in {\bf R}$, 
denote the semiclassical Sobolev space of order
$s$ equipped with the norm $\Vert \langle hD\rangle^s u\Vert$ where
the norms are the ones in $L^2$, $\ell^2$ or the corresponding
operator norms if nothing else
is indicated. Here $\langle hD\rangle= (1+(hD)^2)^{1/2}$. Let
$\widehat{u}(\xi )=\int e^{-ix\cdot \xi }u(x)dx$ denote the Fourier
transform of the tempered distribution $u$ on ${\bf R}^n$. In
\cite{Sj1} we recalled the following result:
\begin{prop}\label{al1}
Let $s>n/2$. Then there exists a constant $C=C(s)$ such that for all
$u,v\in H_h^s({\bf R}^n)$, we have $u\in L^\infty ({\bf R}^n) $, 
$uv\in H_h^s({\bf R}^n)$ and 
\ekv{al.1}
{
\Vert u\Vert_{L^\infty }\le Ch^{-n/2}\Vert u\Vert_{H_h^s},
}
\ekv{al.2}
{
\Vert uv\Vert_{H_h^s} \le Ch^{-n/2} \Vert u\Vert_{H_h^s} \Vert v\Vert_{H_h^s}.
}
\end{prop}

We cover $X$ by
finitely many coordinate neighborhoods $X_1,...,X_p$ and for
each $X_j$, we let $x_1,...,x_n$ denote the corresponding local 
coordinates on $X_j$. Let $0\le \chi _j\in C_0^\infty (X_j)$ have the
property that $\sum_1^p\chi _j >0$ on $X$. Define $H_h^s(X)$ to be the
space of all $u\in {\cal D}'(X)$ such that 
\ekv{al.4}
{
\Vert u\Vert_{H_h^s}^2:=\sum_1^p \Vert \chi _j\langle hD\rangle^s \chi
_j u\Vert ^2 <\infty .
}
It is standard to show that this definition does not depend on the
choice of the coordinate neighborhoods or on $\chi _j$. With different
choices of these quantities we get norms in \no{al.4} which are
uniformly equivalent when $h\to 0$. In fact, this follows from the
$h$-pseudodifferential calculus on manifolds with symbols in the
H\"ormander space $S^m_{1,0}$, that we quickly reviewed in the
appendix in \cite{Sj1}.
An equivalent definition of $H_h^s(X)$ is the following: Let 
\ekv{al.5}
{
h^2\widetilde{R}=\sum (hD_{x_j})^*r_{j,k}(x)hD_{x_k}
}
be a non-negative elliptic operator with smooth coefficients on $X$,
where the star indicates that we take the adjoint with respect to some
fixed positive smooth density on $X$. Then $h^2\widetilde{R}$ is
essentially self-adjoint with domain $H^2(X)$, so
$(1+h^2\widetilde{R})^{s/2}:L^2\to L^2$ is a closed densely defined
operator for $s\in {\bf R}$, which is bounded precisely when $s\le
0$. Standard methods allow to show that $(1+h^2\widetilde{R})^{s/2}$
is an $h$-pseudodifferential operator with symbol in $S^s_{1,0}$ and
semiclassical principal symbol given by $(1+r(x,\xi ))^{s/2}$, where
$r(x,\xi )=\sum_{j,k}r_{j,k}(x)\xi _j\xi _k$ is the semiclassical
principal symbol of $h^2\widetilde{R}$.  See the appendix in
\cite{Sj1}.
The
$h$-pseudodifferential calculus gives for every $s\in {\bf R}$:
\begin{prop}\label{al2}
  $H_h^s(X)$ is the space of all $u\in {\cal D}'(X)$ such that 
$(1+h^2\widetilde{R})^{s/2}u\in L^2$ and the norm $\Vert u\Vert_{H_h^s}$ is
equivalent to $\Vert (1+h^2\widetilde{R})^{s/2}u\Vert$, uniformly when $h\to 0$.
\end{prop}
\begin{remark}\label{al3}
\rm From the first definition we see that Proposition \ref{al1} remains
valid if we replace ${\bf R}^n$ by a compact $n$-dimensional 
manifold $X$.
\end{remark}

\par Of course, $H_h^s(X)$ coincides with the standard Sobolev space
$H^s(X)$ and the norms are equivalent for each fixed value of $h$, but
not uniformly with respect to $h$. The following variant of
Proposition \ref{al1} will probably be useful when studying the high
energy limit (that we hope to treat in a future paper).

\begin{prop}\label{al4}
Let $s>n/2$. Then there exists a constant $C=C_s>0$ such that 
\ekv{al.6}
{
\Vert uv\Vert_{H_h^s}\le C\Vert u\Vert_{H^s}\Vert v\Vert_{H_h^s},\
\forall u\in H^s({\bf R}^n),\, v\in H_h^s({\bf R}^n).
}
The result remains valid if we replace ${\bf R}^n$ by $X$.
\end{prop}
\begin{proof}
The adaptation to the case of a compact manifold is immediate by
working in local coordinates, so it is enough to prove (\ref{al.6}) in
the ${\bf R}^n$-case. 

\par Let $\chi \in C_0^\infty ({\bf R}^n)$ be equal to one in a
neighborhood of $0$. Write $u=u_1+u_2$ with $u_1=\chi (hD)u$,
$u_2=(1-\chi (hD))u$. Then, with hats indicating Fourier transforms, we
have 
$$
\langle h\xi \rangle^s\widehat{u_1v}(\xi )=\frac{1}{(2\pi )^n}
\int \frac{\langle h\xi \rangle^s}{\langle h\eta \rangle^s}
(\chi (h(\xi -\eta ))\widehat{u}(\xi -\eta ))\langle h\eta
\rangle^s\widehat{v}(\eta )d\eta .
$$
Here $\langle h\xi \rangle/\langle h\eta  \rangle={\cal O}(1)$ on the
support of $(\xi ,\eta )\mapsto \chi (h(\xi -\eta ))$, so 
$$
\Vert u_1v\Vert_{H_h^s}\le {\cal O}(1)\Vert
\widehat{u}\Vert_{L^1}\Vert v\Vert_{H_h^s}\le {\cal O}(1)\Vert
u\Vert_{H^s}\Vert v\Vert_{H_h^s},
$$
where we also used that $s>n/2$ in the last estimate.

\par On the other hand, $\langle h\xi \rangle^s\le C h^s\langle
\xi \rangle ^s$ when $1-\chi (h\xi )\ne 0$, so $\Vert
u_2\Vert_{H_h^s}\le Ch^s\Vert u\Vert_{H^s}$. By Proposition \ref{al1},
we get 
$$
\Vert u_2v\Vert_{H_h^s}\le Ch^{-\frac{n}{2}}\Vert
u_2\Vert_{H_h^s}\Vert v\Vert_{H_h^s}\le
\widetilde{C}h^{s-\frac{n}{2}}\Vert u\Vert_{H^s}\Vert
v\Vert_{H_h^s}\le \widetilde{C}\Vert u\Vert_{H^s}\Vert
v\Vert_{H_h^s}, 
$$
when $h\le 1$.
\end{proof}

\section{$H^s$-perturbations and eigenfunctions}\label{hs}
\setcounter{equation}{0}

\par This section gives a very straight forward adaptation of the
corresponding section in \cite{Sj1}. Let $S^m(T^*X)=S^m_{1,0}(T^*X)$,
$S^m(U\times {\bf R}^n)=S^m_{1,0}(U\times {\bf R}^n)$ denote the
classical H\"ormander symbol spaces, where $U\subset {\bf R}^n$ is
open. See for instance \cite{GrSj} and further references given
there. As in \cite{Ha, HaSj}, we can find $\widetilde{p}\in S^m(T^*X)$
which is equal to $p$ outside any given fixed neighborhood of
$p^{-1}(\overline{\Omega })$ such that $\widetilde{p}-z$ is
non-vanishing, for any $z\in \overline{\Omega }$. Let
$\widetilde{P}=P+\mathrm{Op}_h(\widetilde{p}-p)$, where
$\mathrm{Op}_h(\widetilde{p}-p)$ denotes any reasonable quantization
of $(\widetilde{p}-p)(x,h\xi )$. (See for instance the appendix in 
\cite{Sj1}.) Then $\widetilde{P}-z:H^m_h(X)\to H^0_h(X)$ has a
uniformly bounded inverse for $z\in \overline{\Omega }$ and $h>0$
small enough. As in \cite{HaSj, Sj1}, we see that the eigenvalues of
$P$ in $\Omega $, counted with their algebraic multiplicity, coincide
with the zeros of the function $z\mapsto \det
((\widetilde{P}-z)^{-1}(P-z))=\det
(1-(\widetilde{P}-z)^{-1}(\widetilde{P}-P))$.

\par Fix $s>n/2$ and consider the perturbed operator
\ekv{hs.1}{
P_\delta =P+\delta (h^{\frac{n}{2}}q_1+q_2)=P+\delta
(Q_1+Q_2)=P+\delta Q,
}
where $q_j\in H^s(X)$, 
\ekv{hs.2}{
\Vert q_1\Vert_{H^s_h}\le 1,\ \Vert q_2\Vert_{H^s}\le 1,\ 0\le \delta
\ll 1.
}
According to Propositions \ref{al1}, \ref{al4}, $Q={\cal
  O}(1):H_h^s \to H_h^s$ and hence by duality and interpolation,
\ekv{hs.3}
{
Q={\cal O}(1):H_h^\sigma \to H_h^\sigma ,\ -s\le \sigma \le s.
} 

\par As in \cite{Sj1}, the spectrum of $P_\delta $ in $\Omega $
 is discrete and coincides with the set of zeros of 
\ekv{hs.4}
{
\det ((\widetilde{P}_\delta -z)^{-1}(P_\delta -z))=\det (1-
(\widetilde{P}_\delta -z)^{-1}(\widetilde{P}-P)),}
where $\widetilde{P}_\delta :=P_\delta +\widetilde{P}-P$.
Here $(\widetilde{P}-z)^{-1}={\cal O}(1):H_h^\sigma \to H_h^\sigma $
for $\sigma $ in the same range and as in \cite{Sj1} we get the same
conclusion for $(\widetilde{P}_\delta -z)^{-1}$.  

\par
Put 
\ekv{hs.5}
{
P_{\delta ,z}:=(\widetilde{P}_\delta -z)^{-1}(P_\delta -z)=
1-(\widetilde{P}_\delta -z)^{-1}(\widetilde{P}-P)=:1-K_{\delta ,z},
}
\ekv{hs.6}
{
S_{\delta ,z}:=P_{\delta ,z}^*P_{\delta ,z}=1-(K_{\delta ,z}+K_{\delta ,z}^*-K_{\delta ,z}^*K_{\delta ,z})=:1-L_{\delta ,z}.
}
As in \cite{Sj1} we get 
\ekv{hs.7}
{K_{\delta ,z},
L_{\delta ,z}={\cal O}(1):H_h^{-s}\to H_h^s.
}
For $0\le \alpha \le 1/2$, let $\pi _\alpha =1_{[0,\alpha ]}(S_{\delta
,z})$. Then as in \cite{Sj1}, we get
\ekv{hs.8}
{
\pi _\alpha ={\cal O}(1): H_h^{-s}\to H_h^s.
}

\par We also have the corresponding result for $P_\delta -z$. Let 
\ekv{hs.9}
{
S_\delta =(P_\delta -z)^*(P_\delta -z)
}
be defined as the Friedrichs extension from $C^\infty (X)$ with
quadratic form domain $H_h^m(X)$. For $0\le \alpha \le {\cal O}(1)$,
we now put $\pi _\alpha =1_{[0,\alpha ](S_\delta )}$. Then as in
\cite{Sj1}, we see that this new spectral projection also fulfils 
(\ref{hs.9}), for $0\le \alpha \ll 1$.

\section{Some functional and pseudodifferential cal\-culus}\label{fu}
\setcounter{equation}{0}

In this section we derive some results analogous to those of Section 4
in \cite{HaSj}. There we worked on ${\bf R}^n$ and by a simple
dilation and change of the semi-classical parameter from $h$ to
$h/\alpha $ we could reduce ourselves to a situation of more standard
$h/\alpha $-pseudodifferential calculus. On a manifold, this can
probably be done also, but appeared to us as quite heavy, so here we
take another route and develop directly a slightly exotic
pseudodifferential calculus, then use it to study resolvents and
functions of certain self-adjoint pseudodifferential operators.

\par Let $P$ be of the form (\ref{int.1}) and let $p$ in (\ref{int.6})
be the corresponding semi-classical principal symbol. Assume classical
ellipticity as in (\ref{int.3}) and let $z\in {\bf C}$ be fixed
throughout this section. 
Let 
\ekv{fu.1}
{
S=(P-z)^*(P-z),
}
that we realize as a self-adjoint operator in the sense of Friedrichs 
extensions. Later on we will also consider a different choice of $S$, 
namely 
\ekv{fu.1.5}{
S=P_z^*P_z,\hbox{ where }P_z=(\widetilde{P}-z)^{-1}(P-z)} and
$\widetilde{P}$ is defined prior to (\ref{hs.1}). The main goal is to
make a trace class
study of $\chi (\frac{1}{\alpha }S)$ when $0<h\le \alpha \ll 1$, $\chi
\in C_0^\infty ({\bf R})$. With the second choice of $S$, we shall
also study $\ln \det (S+\alpha \chi (\frac{1}{\alpha }S))$, when $\chi
\ge 0$, $\chi (0)>0$. The main step will be to get enough information
about the resolvent $(w-\frac{1}{\alpha }S)^{-1}$ for $w={\cal O}(1)$,
$\Im w\ne 0$ and then apply the Cauchy-Riemann-Green-Stokes formula
\ekv{fu.2}
{
\chi (\frac{1}{\alpha }S)=-\frac{1}{\pi }\int \frac{\partial
  \widetilde{\chi }(w)}{\partial \overline{w}}(w-\frac{1}{\alpha
}S)^{-1}L(dw ),
}
where $\widetilde{\chi }\in C_0^\infty ({\bf C})$ is an almost
holomorphic extension of $\chi $, so that 
\ekv{fu.3}
{
\frac{\partial \widetilde{\chi }}{\partial \overline{w}}={\cal O}(
|\Im w|^\infty ).
}
Thanks to (\ref{fu.3}) we can work in symbol classes with some
temparate but otherwise unspecified growth in $1/|\Im w|$.

\par Let 
\ekv{fu.4}
{
s=|p-z|^2
}
be the semiclassical principal symbol of $S$ in (\ref{fu.1}). A basic
weight function in our calculus will be 
\ekv{fu.5}
{
\Lambda :=\left(\frac{\alpha +s}{1+s} \right)^{\frac{1}{2}},
}
satisfying $\sqrt{\alpha }\le \Lambda \le 1$.

\par As a preparation and motivation for the calculus, we first
consider symbol properties of $1+\frac{s}{\alpha }$ and its powers.
\begin{prop}\label{fu1}
For every choice of local coordinates $x$ on $X$, let $(x,\xi )$
denote the corresponding canonical coordinates on $T^*X$. Then for all 
$\ell\in {\bf R}$, $\widetilde{\alpha },\beta \in {\bf N}^n$, we have 
uniformly in $\xi $ and locally uniformly in $x$:
\ekv{fu.6}
{
\partial _x^{\widetilde{\alpha }}\partial _\xi^\beta
(1+\frac{s}{\alpha })^\ell ={\cal O}(1) (1+\frac{s}{\alpha })^\ell
\Lambda ^{-|\widetilde{\alpha }|-|\beta |}\langle \xi \rangle
^{-|\beta |}.
}
\end{prop} 
\begin{proof}
In the region $|\xi |\gg 1$ we see that $(1+\frac{s}{\alpha })^\ell$
is an elliptic element of the H\"ormander symbol class 
$$
\alpha ^{-\ell}S^{2\ell m}_{1,0}=:\alpha ^{-\ell}S(\langle \xi \rangle
^{2\ell m}),
$$
and $\Lambda \asymp 1$ there, so (\ref{fu.6}) holds. In the region
$|\xi |={\cal O}(1)$, we start with the case $\ell =1$. Since $s\ge
0$, we have $\nabla s={\cal O}(s^{\frac{1}{2}})$, so 
$$
|\nabla (1+\frac{s}{\alpha })|={\cal O}(\frac{s^{\frac{1}{2}}}{\alpha
})\le 
{\cal O}(1)(1+\frac{s}{\alpha })(\alpha +s)^{-\frac{1}{2}}={\cal
  O}(1)(1+\frac{s}{\alpha })\Lambda ^{-1}.
$$
For $k\ge 2$, we have 
$$
|\nabla ^k(1+\frac{s}{\alpha })|={\cal O}(\frac{1}{\alpha })={\cal
  O}(1)(1+\frac{s}{\alpha })\Lambda ^{-2}\le {\cal
  O}(1)(1+\frac{s}{\alpha })\Lambda ^{-k},
$$
and we get (\ref{fu.6}) when $\ell =1$.

\par If $\ell\in {\bf R}$, then $\partial _x^{\widetilde{\alpha
  }}\partial _\xi ^\beta (1+\frac{s}{\alpha })^\ell$ is a finite
linear combination of terms 
$$
(1+\frac{s}{\alpha })^{\ell -k}(\partial _x^{\widetilde{\alpha
  }_1}\partial _\xi^{\beta _1}(1+\frac{s}{\alpha }))\cdots
(\partial _x^{\widetilde{\alpha
  }_k}\partial _\xi^{\beta _k}(1+\frac{s}{\alpha })),  
$$
with $\widetilde{\alpha }=\widetilde{\alpha }_1+...+\widetilde{\alpha
}_k$, $\beta =\beta _1+...+\beta _k$, and we get (\ref{fu.6}) in general.
\end{proof}

\par We next notice that when $w={\cal O}(1)$,
\ekv{fu.7}
{
\frac{|\Im w|}{C}(1+\frac{s}{\alpha })\le |w-\frac{s}{\alpha }|\le
C(1+\frac{s}{\alpha }).
}
In fact, the second inequality is obvious, and so is the first one,
when $\frac{s}{\alpha }\gg 1$. When $\frac{s}{\alpha }\le {\cal
  O}(1)$, it follows from the fact that 
$$
1+\frac{s}{\alpha }={\cal O}(1),\quad |w-\frac{s}{\alpha }|\ge |\Im w|.
$$

\par From (\ref{fu.6}), (\ref{fu.7}), we get
\ekv{fu.8}
{
|\partial _x^{\widetilde{\alpha }}\partial _\xi ^\beta
(w-\frac{s}{\alpha })|\le {\cal O}(1)(w-\frac{s}{\alpha })\Lambda
^{-|\widetilde{\alpha }|-|\beta |}\langle \xi \rangle ^{-|\beta |}|\Im
w|^{-1}.
}
When passing to $(w-\frac{s}{\alpha })^\ell$ and applying the proof of
Proposition \ref{fu1}, we loose more powers of $|\Im w|$ that can
still be counted precisely, but we refrain from doing so and simply
state the following result:
\begin{prop}\label{fu2}
For all $\ell\in {\bf R}$, $\widetilde{\alpha },\beta \in {\bf N}^n$,
there exists $J\in {\bf N}$, such that 
\ekv{fu.9}
{
\partial _x^{\widetilde{\alpha }}\partial _\xi ^\beta
(w-\frac{s}{\alpha })^\ell ={\cal O}(1)(1+\frac{s}{\alpha })^\ell 
\Lambda ^{-|\widetilde{\alpha }|-|\beta |}\langle \xi \rangle^{-|\beta
  |}|\Im w|^{-J},
}
uniformly in $\xi $ and locally uniformly in $x$.
\end{prop}

\par We now define our new symbol spaces.
\begin{dref}\label{fu3}
Let $\widetilde{m}(x,\xi )$ be a weight function of the form 
$\widetilde{m}(x,\xi )=\langle \xi \rangle ^k\Lambda^\ell $. We say
that the family $a=a_w\in C^\infty (T^*X)$, $w\in D(0,C)$, belongs to $S_\Lambda (\widetilde{m})$ if
for all $\widetilde{\alpha },\beta \in {\bf N}^n$ there exists $J\in
{\bf N}$ such that 
\ekv{fu.10}
{
\partial _x^{\widetilde{\alpha }}\partial _\xi ^\beta a=
{\cal O}(1)\widetilde{m}(x,\xi )\Lambda ^{-|\widetilde{\alpha
  }|-|\beta |}\langle \xi \rangle^{-|\beta |}|\Im w|^{-J}.
} 
\end{dref}

Here, as in Proposition \ref{fu2}, it is understood that that the
estimate is expressed in canonical coordinates and is locally uniform
in $x$ and uniform in $\xi $. Notice that the set of estimates
(\ref{fu.10}) is invariant under changes of local coordinates in $X$.

\par Let $U\subset X$ be a coordinate neighborhood that we shall view
as a subset of ${\bf R}^n$ in the natural way. Let $a\in S_\Lambda
(T^*U,\widetilde{m})$ be a symbol as in Definition \ref{fu3} so that
(\ref{fu.10}) holds uniformly in $\xi $ and locally uniformly in
$x$. For fixed values of $\alpha $, $w$ the symbol $a$ belongs
to $S^k_{1,0}(T^*U)$, so the classical $h$-quantization
\ekv{fu.11}
{
Au=\mathrm{Op}_h(a)u(x)=\frac{1}{(2\pi h)^n}\iint
e^{\frac{i}{h}(x-y)\cdot \eta }a(x,\eta ;h)u(y)dyd\eta 
}
is a well-defined operator $C_0^\infty (U)\to C^\infty (U)$, ${\cal
  E}'(U)\to {\cal D}'(U)$. In order to develop our rudimentary
calculus on $X$ we first establish a pseudolocal property for the
distribution kernel $K_A(x,y)$:
\begin{prop}\label{fu4}
For all $\widetilde{\alpha },\beta \in {\bf N}^n$, $N\in {\bf N}$,
there exists $M\in {\bf N}$ such that 
\ekv{fu.12}
{
\partial _x^{\widetilde{\alpha }}\partial _y^\beta K_A(x,y)={\cal
  O}(h^N|\Im w|^{-M}),
}
locally uniformly on $U\times U\setminus \mathrm{diag}(U\times U)$.
\end{prop}
\begin{proof}
If $\gamma \in {\bf N}^n$, then $(x-y)^\gamma K_A(x,y)$ is the
distribution kernel of $\mathrm{Op}_h((-hD_\xi )^\gamma a)$ and 
$(-hD_\xi )^\gamma a\in S_\Lambda
\left(\widetilde{m}\left(\frac{h}{\Lambda \langle \xi \rangle}\right)^{|\gamma |}\right)
$ and we notice that $h/\Lambda \le h/\alpha ^{\frac{1}{2}}\le
h^{\frac{1}{2}}$. Thus for any $N\in {\bf N}$, we have 
$$
(x-y)^\gamma K_A(x,y)={\cal O}(h^N|\Im w|^{-M})\hbox{ if }|\gamma |\ge
\gamma (N)
$$ 
is large enough. From this we get (\ref{fu.12}) when
$\widetilde{\alpha }=\beta =0$. Now, $\partial _x^{\widetilde{\alpha
  }}\partial _y^\beta K_A$ can be viewed as the distribution kernel of
a new pseudodifferential operator of the same kind, so we get
(\ref{fu.12}) for all $\widetilde{\alpha },\beta $.
\end{proof}

This means that if $\phi ,\psi \in C_0^\infty (U)$ have disjoint
supports, then for every $N\in {\bf N}$, there exists $M\in {\bf N}$
such that $\phi A\psi :H^{-N}({\bf R}^n)\to H^N({\bf R}^n)$ with norm
${\cal O}(h^N|\Im w|^{-M})$, and this leads to a simple way of
introducing pseudo\-differential operators on $X$: Let $U_1,...,U_s$ be
coordinate neighborhoods that cover $X$. Let $\chi _j\in C_0^\infty
(U_j)$ form a partition of unity and let $\widetilde{\chi }_j\in
C_0^\infty (U_j)$ satisfy $\chi _j\prec \widetilde{\chi }_j$ in the
sense that $\widetilde{\chi }_j$ is equal
to 1 near $\mathrm{supp\,}(\chi _j)$. Let $a=(a_1,...,a_s)$, where
$a_j\in S_\Lambda (\widetilde{m})$. Then we quantize $a$ by the
formula:
\ekv{fu.13}
{
A=\sum_1^s \widetilde{\chi }_j\circ \mathrm{Op}_h(a_j)\circ \chi _j.
}
This is not an invariant quantization procedure but it
will suffice for our purposes. 

We next study the composition to the left with non-exotic
pseudodifferential operators. Let $U$ be a coordinate neighborhood,
viewed as an open set in ${\bf R}^n$, and take $A= \mathrm{Op}_h(a)$, 
$a\in S_{1,0}(m_1)$, $m_1=\langle \xi \rangle^r$,
$B=\mathrm{Op}_h(b)$, $b\in S_\Lambda (m_2)$ with $m_2=\langle \xi
\rangle^k\Lambda ^\ell$ as in Definition \ref{fu3}. We will assume
that $\mathrm{supp\,}(b)\subset K\times {\bf R}^n$, where $K\subset U$
is compact. We are interested in $C=A\circ B$.

\par The symbol $c$ of this composition is given by
\eekv{fu.14}
{
&&c(x,\xi ;h)=e^{-\frac{i}{h}x\cdot \xi }A(b(\cdot ,\xi
)e^{\frac{i}{h}(\cdot )\cdot \xi })(x)}{&&=\frac{1}{(2\pi h)^n}\iint
a(x,\eta )b(y,\xi )e^{\frac{i}{h}(x-y)\cdot (\eta -\xi )}dyd\eta 
}  

\par In the region $|\eta -\xi |\ge \frac{1}{C}\langle \xi \rangle$ we can
make repeated integrations by parts in the $y$-variables and see that
the contribution from this region is a symbol $d(x,\xi ;h)$ satisfying 
\eekv{fu.15}
{
\forall N\in {\bf N}, \widetilde{\alpha }, \beta \in {\bf N}^n,
\exists  M\in {\bf N}, \forall K\Subset U, \exists C>0;}
{|\partial
_x^{\widetilde{\alpha} }\partial _{\xi }^\beta d(x,\xi ;h)|\le
C\frac{h^N\langle \xi \rangle^{-N}}{|\Im w|^M},\ (x,\xi )\in K\times
{\bf R}^n.
}

\par Up to such a term $d$, we may assume that with $\chi \in
C_0^\infty (B(0,\frac{1}{2}))$ equal to 1 near 0, 
\eekv{fu.16}
{c(x,\xi ;h)&\equiv&\frac{1}{(2\pi h)^n}\iint
a(x,\eta )b(y,\xi )\chi (\frac{\eta -\xi }{\langle \xi \rangle})e^{\frac{i}{h}(x-y)\cdot (\eta -\xi )}dyd\eta
}
{&=&\left(\frac{\langle \xi \rangle}{2\pi h}\right)^n
\iint
a(x,\langle \xi \rangle(\eta +\frac{\xi }{\langle \xi \rangle} ))
b(x+y,\xi )\chi (\eta )e^{\frac{-i\langle \xi
    \rangle}{h}y\cdot \eta }dyd\eta .
}
The method of stationary phase gives for every $N\in {\bf N}$:
\ekv{fu.17}
{
c(x,\xi ;h)=\sum_{|\beta |<N}\frac{h^{|\beta |}}{\beta !}\partial _\xi
^\beta a D_x^\beta b +R_N. 
}
Here,
\eekv{fu.18}
{
R_N&=&\left(\frac{h}{\langle \xi \rangle}\right)^N\frac{1}{(N-1)!}\times }
{&&\int_0^1(1-t)^NJ\left( t\frac{h}{\langle \xi \rangle},
(\partial _\eta \cdot D_y)^N(a(x,\langle \xi
\rangle(\eta +\frac{\xi }{\langle \xi \rangle}))b(x+y,\xi )\chi (\eta
))
\right) dt,
}
where 
$$
J(s,u)=\frac{1}{(2\pi s)^n}\iint u(y,\eta )e^{-\frac{i}{s}y\cdot \eta
}dyd\eta ,
$$
and we used the fact that $\partial _sJ(s,u)=J(s,\partial _\eta \cdot
D_y(u))$, $J(s,u)\to u(0)$ when $s\to 0$.
Noting that 
$$
|J(s,u)|\le C\sum_{|\widetilde{\alpha }|+|\beta |\le
  2n+1}\Vert \partial _y^{\widetilde{\alpha }}\partial _\eta ^\beta u\Vert_{L^1},
$$
we see that there exist exponents $N_2,N_3$ independent of $N$, such
that 
$$
|R_N|\le C\left(\frac{h}{\langle \xi \rangle}\right)^N m_1(\xi )
\langle \xi \rangle^{N_2}\alpha ^{N_3-\frac{N}{2}}|\Im w|^{-M(N)}.
$$
Similar estimates hold for the derivatives and we conclude:
\begin{prop}\label{fu5}
Let $A=\mathrm{Op}_h(a)$, $a\in S_{1,0}(m_1)$, $B=\mathrm{Op}_h(b)$,
$b\in S_\Lambda (m_2)$ and assume that $b$ has uniformly compact
support in $x$. Then $A\circ B=\mathrm{Op}_h(c)$, where $c$ belongs to
$S_\Lambda
(m_1m_2)$ and has the asymptotic expansion
$$
c\sim \sum \frac{h^{|\beta |}}{\beta !}\partial _\xi ^\beta a(x,\xi )D_x^\beta
b(x,\xi ),
$$
in the sense that for every $N\in {\bf N}$,
$$
c= \sum_{|\beta |<N} \frac{h^{|\beta |}}{\beta !}\partial _\xi ^\beta a(x,\xi )D_x^\beta
b(x,\xi ) + r_N(x,\xi ;h),
$$
where 
$r_N\in S_\Lambda (\frac{m_1m_2}{(\Lambda \langle \xi \rangle)^N}h^N)$.
\end{prop}

\par We next make a parametrix construction for $w-\frac{1}{\alpha
}S$, still with $S$ as in (\ref{fu.1}), and most of the work will take
place in a coordinate neighborhood $U$, viewed as an open set in ${\bf
  R}^n$. The symbol of $w-\frac{1}{\alpha }S$ is of the form
\ekv{fu.19}
{
F=F_0+F_{-1},\quad F_0=w-\frac{1}{\alpha }s,\ F_{-1}=\frac{h}{\alpha }s_{-1}\in
S(\frac{h}{\alpha }\langle \xi \rangle^{2m-1}).
} 
Put 
\ekv{fu.20}
{
E_0=\frac{1}{w-\frac{1}{\alpha }s}\in S_\Lambda (\frac{\alpha
}{\Lambda ^2\langle \xi \rangle^{2m}}).
}
With Proposition \ref{fu5} in mind, we first consider the formal
composition
\eekv{fu.21}
{F\# E_0&\sim& \sum \frac{h^{|\beta |}}{\beta !}(\partial _\xi ^\beta F)
  (D_x^{\beta }E_0)}
{&\sim& 1+\sum_{|\beta |\ge 1}\frac{h^{|\beta |}}{\beta !}(\partial
  _\xi ^\beta F_0) (D_x^\beta E_0)+F_{-1}\# E_0.
}
Here,
$$
F_{-1}\# E_0\in S_\Lambda (\frac{h}{\alpha }\langle \xi
\rangle^{2m-1}\frac{\alpha }{\Lambda ^2\langle \xi
  \rangle^{2m}})=S_\Lambda (\frac{h}{\Lambda ^2\langle \xi \rangle }).
$$
Since $F_0$ also belongs to $S_\Lambda (\frac{1}{\alpha }\Lambda
^2\langle \xi \rangle^{2m})$, we see that for $|\beta |\ge 1$,
$$
h^{|\beta |}(\partial _\xi ^\beta F_0)(D_x^\beta E_0)\in S_\Lambda
(\frac{h^{|\beta |}}{\Lambda ^{2|\beta |}\langle \xi \rangle^{|\beta
  |}})\subset
S_\Lambda (\frac{h}{\Lambda ^2\langle \xi \rangle}),
$$
and this can be improved for $|\beta |\ge 2$, using that
$F\in S_{1,0}(\frac{1}{\alpha }\langle \xi \rangle^{2m})$. Hence,
$$
F\# E_0=1+r_1,\ r_1\in S_\Lambda (\frac{h}{\Lambda ^2\langle \xi \rangle}).
$$

\par Now put $E_1=E_0-r_1/(w-s/\alpha )$. Then by the same estimates
with an extra power of $h\Lambda ^{-2}\langle \xi \rangle^{-1}$, we
get
$$
F\# E_1=1+r_2,\ r_2\in S_{\Lambda }((\frac{h}{\Lambda ^2\langle \xi \rangle})^2),
$$
and iterating the procedure we get
\ekv{fu.22}
{
E_N\equiv \frac{1}{w-\frac{s}{\alpha }}\ \mathrm{mod}\ S_\Lambda
(\frac{\alpha }{\Lambda ^2\langle \xi \rangle ^{2m}}\frac{h}{\Lambda
  ^2\langle \xi \rangle}),
}
such that
\ekv{fu.23}
{
F\# E_N=1+r_N,\ r_N\in S_\Lambda ((\frac{h}{\Lambda ^2\langle \xi 
\rangle})^{N+1}).
}
Actually, in this construction we can work with finite sums instead
of asymptotic ones and then
\ekv{fu.24}
{
E_N\hbox{ is a holomorphic function of }w,\hbox{ for }|\xi |\ge C,
}
where $C$ is independent of $N$.

\par Now we return to the manifold situation and denote by
$E_N^{(j)}$, $r_N^{(j)}$ the corresponding symbols on $T^*U_j$,
constructed above. Denote the operators by the same symbols, and put
on the operator level:
\ekv{fu.25}
{
E_N=\sum_{j=1}^s \widetilde{\chi }_jE_N^{(j)}\chi _j,
}
with $\chi _j$, $\widetilde{\chi _j}$ as in (\ref{fu.13}). Then
\eeekv{fu.26}
{
(w-\frac{1}{\alpha }S)E_{N-1}&=&1-\sum_{j=1}^s \frac{1}{\alpha
}[S,\widetilde{\chi }_j]E_{N-1}^{(j)}\chi _j+\sum_{j=1}^s\widetilde{\chi
}_j
r_N^{(j)}\chi _j
}{&=:&1+R_N^{(1)}+R_N^{(2)}
}
{&=:&1+R_N.}
Proposition \ref{fu4} implies that for every $\widetilde{N}$, there
exists an $\widetilde{M}$ such that the trace class norm of
$R_N^{(1)}$ satisfies
\ekv{fu.27}
{
\Vert R_N^{(1)}\Vert_{\mathrm{tr}}\le {\cal O}(h^{\widetilde{N}}|\Im
w| ^{-\widetilde{M}}).
}

\par As for the trace class norm of $R_N^{(2)}$, we review some easy
facts about such norms for pseudodifferential operators:

\par If $A=a(x,D)$ is a pseudodifferential operator on ${\bf R}^n$, 
either in the
Weyl or in the classical quantization, then $A$ is of trace class and
we have 
$$
\Vert A\Vert_{\mathrm{tr}}\le C\iint \sum_{_{|\beta |\le
    2n+1}}|\partial _{x,\xi }^\beta a | dxd\xi ,
$$ 
provided that the integral is finite. In that case we also know that
$$
\mathrm{tr\,}(A)=\frac{1}{(2\pi )^n}\iint a(x,\xi )dxd\xi .
$$
 See Robert \cite{Ro}, and also
$\cite{DiSj}$ for a sharper statement. If instead we consider an
$h$-pseudodifferential operator $A=a(x,hD)$, then it is unitarily
equivalent to $\widetilde{A}=a(h^{\frac{1}{2}}x,h^{\frac{1}{2}}D_x)$,
so 
$$
\Vert A\Vert_{\mathrm{tr}}\le \frac{C}{h^n}\iint \sum_{_{|\beta |\le
    2n+1}}|(h^{\frac{1}{2}}\partial _{x,\xi })^\beta a | dxd\xi ,
$$
where the factor $h^{-n}$ is the Jacobian, when passing from 
$h^{1/2}x, h^{1/2}\xi $ to $x,\xi $.

\par Now, let $a\in S_\Lambda (m)$ be a symbol on $T^*U$ with
uniformly compact support in $x$. Then for $|\beta |\le 2n+1$, we
have
$$
h^{\frac{|\beta |}{2}}\partial _{x,\xi }^\beta a={\cal
  O}(1)m\left(\frac{h}{\alpha }
\right)^{\frac{|\beta |}{2}}|\Im w|^{-M(\beta )}. 
$$
Thus there exists $M\ge 0$ such that $a(x,hD_x)$ is of trace class and
\ekv{fu.28}
{
\Vert a(x,hD)\Vert_{\mathrm{tr}}\le Ch^{-n}\iint_{U\times {\bf R}^n}
m(x,\xi )dxd\xi\, |\Im w|^{-M},
}
provided that the integral converges. 

\par From (\ref{fu.26}), (\ref{fu.23}), we now get
$$
\Vert R_N^{(2)}\Vert_{\mathrm{tr}}\le Ch^{-n}|\Im w|^{-M(N)}\iint
\left(\frac{h}{\Lambda ^2\langle \xi \rangle} \right)^N dxd\xi ,
$$
and (\ref{fu.27}) then shows that we have the same estimate for $R_N$:
\ekv{fu.29}
{
\Vert R_N\Vert_{\mathrm{tr}}\le Ch^{-n}|\Im w|^{-M(N)}\iint
\left(\frac{h}{\Lambda ^2\langle \xi \rangle} \right)^N dxd\xi .
}
The contribution to this expression from the region where $\Lambda \ge
1/C$ is ${\cal O}(h^{N-n})|\Im w|^{-M(N)}$. 

\par The volume growth assumption (\ref{int.6.2}), that we now assume
for our fixed $z$, says that
\ekv{fu.30}
{
V(t):=\mathrm{vol\,}(\{ \rho \in T^*X;\, s\le t \})={\cal O}(t^\kappa
),\ 0\le t\ll 1,
}
for $0<\kappa \le 1$. The contribution to the integral in
(\ref{fu.29}) from the region $0\le s\le t_0$, $0<t_0\ll 1$, is equal
to some negative power of $|\Im w|$ times
\begin{eqnarray*}
&&{\cal O}(1)\int_0^{t_0}\left(\frac{h}{\alpha +t} \right)^N dV(t)\\
&=&
{\cal O}(1)\left[ \left(\frac{h}{\alpha +t}
  \right)^NV(t) \right]_{t=0}^{t_0}+{\cal
  O}(1)\int_0^{t_0}\frac{h^N}{(\alpha +t) ^{N+1}}V(t)dt\\
&=& {\cal O}(1)h^N+{\cal O}(1)h^N\int_0^{t_0}\frac{t^\kappa }{(\alpha +t)^{N+1}}dt.
\end{eqnarray*} 
The last integral is equal to
$$
\int _0^{t_0/\alpha }\frac{(\alpha s)^\kappa }{\alpha ^{N+1}(1+s)^{N+1}}\alpha
ds\le \alpha ^{\kappa -N}\int_0^\infty \frac{s^\kappa }{(1+s)^{N+1}}ds.
$$
Thus,
\ekv{fu.31}
{
\Vert R_N\Vert_{\mathrm{tr}}\le {\cal O}(1) h^{-n}\alpha ^{\kappa
}\left(\frac{h}{\alpha } \right)^N|\Im w|^{-M(N)}.
}

\par From (\ref{fu.26}), we get 
$$
(w-\frac{1}{\alpha }S)^{-1}=E_{N-1}-(w-\frac{1}{\alpha }S)^{-1}R_N.
$$
Write 
$$
E_{N-1}=\frac{1}{w-\frac{s}{\alpha }}+F_{N-1},\quad F_{N-1}\in S_\Lambda
(\frac{\alpha h}{\Lambda ^4\langle \xi \rangle^{2m+1}}).
$$
More precisely we do this for each $E_{N-1}^{(j)}$ in (\ref{fu.25}). Then
quantize and plug this into (\ref{fu.2}):
\eekv{fu.32}
{
\chi (\frac{1}{\alpha }S)&=&-\frac{1}{\pi} \int \frac{\partial
    \widetilde{\chi }}{\partial
    \overline{w}}\mathrm{Op}_h(\frac{1}{w-\frac{s}{\alpha }})L(dw)
-\frac{1}{\pi }\int \frac{\partial
    \widetilde{\chi }}{\partial
    \overline{w}}F_{N-1} L(dw)
}
{&&-\frac{1}{\pi} \int \frac{\partial
    \widetilde{\chi }}{\partial
    \overline{w}}(w-\frac{1}{\alpha }S)^{-1}R_N L(dw)=:\mathrm{I}+\mathrm{II}+\mathrm{III}.
}
Here by definition,
$$
\mathrm{Op}_h\left( \frac{1}{w-\frac{s}{\alpha
  }}\right)=\sum_{j=1}^s\widetilde{\chi }_j\mathrm{Op}_h\left(\frac{1}{w-\frac{s}{\alpha
  }} \right)\chi _j
$$
with the coordinate dependent quantization appearing to the right. 
$$
\mathrm{tr\,}(-\frac{1}{\pi }\int \frac{\partial \widetilde{\chi
  }}{\partial \overline{w}}(w)\widetilde{\chi }_j \mathrm{Op}_h\left( \frac{1}{w-\frac{s}{\alpha
  }}\right) \chi _j L(dw))
$$
is equal to 
\begin{eqnarray*}
\frac{1}{(2\pi h)^n}\iint -\frac{1}{\pi }\int 
\frac{\partial \widetilde{\chi }}
{\partial \overline{w}}(w) \frac{1}{w-\frac{s}{\alpha
  }}L(dw)\chi _j(x)dxd\xi \\ =\frac{1}{(2\pi h)^n}\iint \chi (\frac{s(x,\xi
  )}{\alpha })\chi _j(x)dxd\xi ,
\end{eqnarray*}
so 
\ekv{fu.33}{
\mathrm{tr\,}(\mathrm{I})=\frac{1}{(2\pi h)^n}\iint \chi (\frac{s(x,\xi
  )}{\alpha })dxd\xi .
}
As at the last estimate in the proof of Proposition 4.4 in 
\cite{HaSj} we see that this quantity is ${\cal O}(\alpha
^\kappa h^{-n})$ and more generally, 
$$
\Vert \mathrm{I}\Vert_{\mathrm{tr}}={\cal O}(\alpha ^\kappa h^{-n}).
$$

\par For II, we get, using the fact that the symbol is holomorphic in
$w$ for large $\xi $,
\begin{eqnarray*}
\Vert \mathrm{II}\Vert_{\mathrm{tr}}&=&{\cal O}(h^{-n})\iint_{|\xi |\le
  C}\frac{h\alpha }{(\alpha +s)^2}dxd\xi \\
&=& {\cal O}(1)h^{-n}\frac{h}{\alpha
}\int_0^{t_0}\frac{1}{(1+\frac{t}{\alpha })^2}dV(t)\\
&=&{\cal O}(1)h^{-n}\frac{h}{\alpha}\left( \left[ (1+\frac{t}{\alpha
  })^{-2}V(t)
\right]_0^{t_0}+\int_0^{t_0}(1+\frac{t}{\alpha })^{-3}V(t) \frac{dt}{\alpha
} 
\right)\\
&=&{\cal O}(1)h^{-n}\frac{h}{\alpha }(\alpha ^2+\alpha ^\kappa
\int_0^{t_0}
(1+\frac{t}{\alpha })^{-3}\left(\frac{t}{\alpha } \right)^\kappa
\frac{dt}{\alpha })\\
&=&{\cal O}(1)\frac{\alpha ^\kappa }{h^n}\frac{h}{\alpha }.
\end{eqnarray*}

\par It is also clear that 
$$\Vert \mathrm{III}\Vert_{tr}={\cal O}(1)\frac{\alpha ^\kappa
}{h^n}\left(\frac{h}{\alpha } \right)^N.$$

Summing up our estimates, we get under the assumption (\ref{fu.30})
(equivalent to (\ref{int.6.2})) the following result:
\begin{prop}\label{fu6}
Let $\chi \in C_0^\infty ({\bf R})$. For $0<h\le \alpha <1$, we have
\ekv{fu.34}
{
\Vert \chi (\frac{1}{\alpha }S)\Vert_{\mathrm{tr}}={\cal
  O}(1)\frac{\alpha ^\kappa }{h^n},
}
\ekv{fu.35}
{
\mathrm{tr\,}\chi (\frac{1}{\alpha }S)=\frac{1}{(2\pi h)^n}
\iint \chi (\frac{s(x,\xi )}{\alpha })dxd\xi +{\cal O}(\frac{\alpha
  ^\kappa }{h^n}\frac{h}{\alpha }).
}
\end{prop}

\begin{remark}\label{fu7}
Using simple $h$-pseudodifferential calculus (for instance as in
the appendix of \cite{Sj1}, we see that if we redefine $S$ as in
(\ref{fu.1.5}), then in each local coordinate chart,
$S=\mathrm{Op}_h(S)$, where $S\equiv s\,\mathrm{mod\,}S_{1,0}(h\langle
\xi \rangle^{-1})$ and $s$ is now redefined as
\ekv{fu.36}
{
s(x,\xi )=\left(\frac{|p(x,\xi )-z|}{|\widetilde{p}(x,\xi )-z|} 
\right)^2.
} 
The discussion goes through without any changes (now with $m=0$) and
we still have Proposition \ref{fu6} with the new choice of $S$, $s$.
\end{remark}

In the remainder of this section, we choose $S$, $s$ as in
(\ref{fu.1.5}), (\ref{fu.36}). In this case we notice that $S$ is a
trace class perturbation of the identity, whose symbol is $1+{\cal
  O}(h^\infty /\langle \xi \rangle^\infty )$ and similarly for all its
derivatives, in a region $|\xi |\ge \mathrm{Const}$.

\par Let $0\le \chi \in C_0^\infty ([0,\infty [)$ with $\chi (0)>0$ and
let $\alpha _0>0$ be small and fixed. Using standard
pseudodifferential calculus in the spirit of \cite{MeSj}, we get
\ekv{fu.37}
{
\ln\det (S+\alpha _0\chi (\frac{1}{\alpha _0}S))=
\frac{1}{(2\pi h)^n}(\iint \ln (s+\alpha _0\chi (\frac{1}{\alpha
  _0}s))dxd\xi +{\cal O}(h)).
}

\par Extend $\chi $ to be an element of $C_0^\infty ({\bf R};{\bf C})$
in such a way that $t+\chi (t)\ne 0$ for all $t\in {\bf R}$. As in
\cite{HaSj}, we use that 
\ekv{fu.39}
{
\frac{d}{dt}\ln (E+t\chi (\frac{E}{t}))=\frac{1}{t}\psi (\frac{E}{t}),
}
where
\ekv{fu.39.5}
{
\psi (E)=\frac{\chi (E)-E\chi '(E)}{E+\chi (E)},
}
so that $\psi \in C_0^\infty ({\bf R})$. By standard functional
calculus for self-adjoint operators, we have 
\ekv{fu.40}
{
\frac{d}{dt}\ln \det (S+t\chi (\frac{S}{t}))=\mathrm{tr\,}
\frac{1}{t}\psi (\frac{S}{t}).
}
Using (\ref{fu.35}), we then get for $t\ge \alpha \ge h>0$:
$$
\frac{d}{dt}\ln\det (S+t\chi (\frac{1}{t}S))
=\frac{1}{(2\pi h)^n}(\iint \frac{1}{t}\psi (\frac{s}{t})dxd\xi +{\cal
  O}(ht^{\kappa -2})).
$$
Integrating this from $t=\alpha _0$ to $t=\alpha $ and using
(\ref{fu.37}), (\ref{fu.39}), we get 
\ekv{fu.41}
{
\ln \det (S+\alpha \chi (\frac{1}{\alpha }S))=
\frac{1}{(2\pi h)^n}(\iint \ln (s+\alpha \chi (\frac{s}{\alpha
}))dxd\xi +{\cal O}(\frac{h}{\alpha })r_\kappa (\alpha )),
}
where $r_\kappa (\alpha )=\alpha ^\kappa $ when $\kappa <1$, and
$r_1(\alpha )=\alpha \ln \alpha $.

\par Improving the calculation prior to (4.22) in \cite{HaSj}, we get
\begin{eqnarray*}
\iint \ln (s+\alpha \chi (\frac{s}{\alpha }))dxd\xi &=&\iint \ln
(s)dxd\xi +\int_0^\alpha \iint \frac{1}{t}\psi (\frac{s}{t})dxd\xi \,
dt\\
&=&\iint \ln
(s)dxd\xi +\int_0^\alpha t^{\kappa -1} dt\\
&=&\iint \ln (s)dxd\xi +{\cal O}(\alpha ^\kappa ).
\end{eqnarray*}
and together with (\ref{fu.41}) this leads to 
\begin{prop}\label{fu8}
If $0\le \chi \in C_0^\infty ([0,\infty [)$, $\chi (0)>0$, we have
uniformly for $0<h\le \alpha \ll 1$
\ekv{fu.42}
{
\ln\det (S+\alpha \chi (\frac{1}{\alpha }S))=\frac{1}{(2\pi h)^n}
(\iint \ln s(x,\xi )dxd\xi +{\cal O}(\alpha ^\kappa \ln \alpha )).
}
Here the remainder term can be replaced by ${\cal O}(\alpha ^\kappa )$
when $\kappa <1$ and by ${\cal O}(\alpha +h\ln \alpha )$ when $\kappa =1$.
\end{prop}

\section{End of the proof}\label{en}
\setcounter{equation}{0}

Having established in Section \ref{fu} the analogues for manifolds of
the results in Section 4 of \cite{HaSj}, the remainder of the proof of
Theorem \ref{int1} is basically identical to the proof  in \cite{Sj1} 
for the ${\bf R}^n$ case. For that reason, we will only give a brief 
outline.

\par Let $P$ be as in (\ref{int.1}), (\ref{int.2}), classically
elliptic as in (\ref{int.3}) and let $p$ be the semiclassical
principal symbol. Assume (\ref{int.5}). To start with, we let $z\in \Omega $ be fixed and
assume (\ref{fu.30}), where $s=|p-z|^2$ is the semiclassical principal 
symbol of $S=(P-z)^*(P-z)$. Also, put 
\ekv{en.1}
{
S_z =P_z^*P_z,\ P_z=(\widetilde{P}-z)^{-1}(P-z).
}  

\par We fix the choice of an operator $P_0=P_{\delta _0}$ as in
(\ref{hs.1}), (\ref{hs.2}) (with $\delta =\delta _0$ still depending
on $h$) and define $\widetilde{P}_{\delta _0}$, $P_{\delta _0,z}$,
$S_{\delta_0 ,z}$, $S_{\delta _0}$ as in that section. As in
\cite{Sj1}, ${\cal D}(S_{\delta _0})=\{ u\in H^m(X);\, P_{\delta
  _0}u\in H^m(X)\}$. As there, we also introduce the self-adjoint
operator $T_{\delta _0}=(P_{\delta _0}-z)(P_{\delta _0}-z)^*$ with
domain ${\cal D}(T_{\delta _0})=\{ u\in H^m(X);\, P^*_{\delta
  _0}u\in H^m(X)\}$. In some fixed ($h$-independent) neighborhood of
$0$ the spectra of 
$S_{\delta _0}$, $T_{\delta _0}$ are discrete and coincide. If $0<\alpha \ll 1$, denote the (common)
eigenvalues in $[0,\alpha [$ by $t_1^2,t_2^2,...,t_N^2$, where $0\le
t_1\le t_2\le ...\le t_N$. Then, there are orthonormal families of
eigenfunctions, $e_1,...,e_N\in {\cal D}(S_{\delta _0})$,
$f_1,...,f_N\in {\cal D}(T_{\delta _0})$ such that
\ekv{en.2}{(P_{\delta _0}-z)e_j=t_jf_j,\ (P_{\delta
    _0}-z)^*f_j=t_je_j.}

\par Define $R_+:L^2(X)\to {\bf C}^N$, $R_-:{\bf C}^N\to L^2(X)$ by
\ekv{en.3}
{
R_+u(j)=(u|e_j),\ R_-u_-=\sum_{1}^N u_-(j)f_j.
}
The Grushin problem
\ekv{en.4}
{(P_{\delta _0}-z)u+R_-u_-=v,\ R_+u=v_+,}
has a unique solution $u=E^0v+E_+^0v_+\in H^m(X)$, $u_-=E^0_-v+E^0_{-+}v_+\in
{\bf C}^N$ for every $(v,v_+)\in L^2(X)\times {\bf C}^N$, and
$E^0_\pm$ and $E^0_{-+}$ can be given explicitly. In particular, 
$E^0_{-+}=-\mathrm{diag\,}(t_j)$.

\par Let now $P_\delta =P_{\delta _0}+\delta Q$ be a small perturbation
(in a suitable sense) of $P_{\delta _0}$. Then we still have a
wellposed problem after replacing $P_{\delta _0}$ by $P_\delta $ in
(\ref{en.4}) with the solution $u=E^{\delta}v+E_+^{\delta}v_+\in H^m$, $u_-=E^{\delta}_-v+
E^{\delta}_{-+}v_+\in {\bf C}^N$ and the new solution operators have Neumann
series expansions. In particular, 
\ekv{en.5}
{
E_{-+}^\delta = E_{-+}^0-\delta E_-^\delta QE_+^0+\delta ^2E_-^0QE^0QE_+^0
-...,}
where we can write the leading perturbation 
$-\delta E_-^0Qe_+^0=-\delta M$, where $M=(M_{j,k})_{1\le j,k\le N}$, 
$M_{j,k}=(Qe_k|f_j)$. When $Q$ is a multiplication operator,
$Qu(x)=q(x)u(x)$, then 
\ekv{en.6}
{
M_{j,k}=\int q(x)e_k(x)\overline{f}_j(x)dx.
}

\par Now, adopt the symmetry assumption (\ref{int'.7}). Then we can
replace the orthonormal family $f_j$ by the new orthonormal family of
eigenfunctions $\widetilde{f}_j=\overline{e}_j$ without changing the
singular values of $E_{-+}^0$, $E_{-+}^\delta $ and we get 
\ekv{en.7}
{
M_{j,k}=\int q(x)e_k(x)e_j(x) dx.
}
In \cite{Sj1} we showed how to find admissible potentials $q$ as in
(\ref{int.6.3}), (\ref{int.6.4}), such
that $M$ gets at least $N/2$ ``large'' singular values and this was
used in an iteration procedure in order to find perturbations of the
form $P_\delta $ where $q$ is an admissible potential, for which the
small singular values are not ``too small''. 

\par Strengthen the assumption on $\delta _0$ to 
\ekv{en.8}
{
\delta _0\le h.
}
Then combining Proposition \ref{fu6} with the perturbative functional
calculus in Section 4 of \cite{Sj1}, we obtain that for $0<h\le \alpha
\ll 1$, the number of eigenvalues of $S_{\delta _0}$ in $[0,\alpha ]$
satisfies $N={\cal O}(\alpha ^\kappa h^{-n})$. The iteration scheme in
Sections 5 to 7 in \cite{Sj1} now works without any changes and we
get the following analogue of Proposition 7.3 there:
\begin{prop}\label{en1}
We make the assumptions above (with $z$ fixed). Let
$s>\frac{n}{2}$, $0<\epsilon <s-\frac{n}{2}$, $N_1=\widetilde{M}+sM+
\frac{n}{2}$, $N_2=
2(N_1+n)+\epsilon _0$, where $\epsilon _0>0$ and $M,\widetilde{M}$ are
as in (\ref{int.6.4}). Let $L$, $R$ be 
$h$-dependent parameters as in (\ref{int.6.4}).
Let $0<\tau_0\le \sqrt{h}$ and let $N^{(0)}={\cal O}(h^{\kappa -n})$ 
be the number of singular
values of $P_{\delta _0}-z$ in $[0,\tau_0[$. Let $0<\theta <\frac{1}{4}$ and let
$N(\theta )\gg 1$ be sufficiently large. Define $N^{(k)}$, $1\le k\le
k_1$ iteratively in the following way. As long as $N^{(k)}\ge
N(\theta )$, we put $N^{(k+1)}=[(1-\theta )N^{(k)}]$ (the integer part
of $(1-\theta )N^{(k)}$). Let $k_0\ge 0$
be the last $k$ value we get in this way. For $k>k_0$ put $N^{(k+1)}=N^{(k)}-1$ until
we reach the value $k_1$ for which $N^{(k_1)}=1$.

\par Put $\tau_0^{(k)}=\tau_0h^{kN_2}$, $1\le k\le k_1+1$. Then there 
exists $q=q_h(x)$ of the
form \no{int.6.3}, satisfying \no{int.6.4}, so that by the
choice of $L$,
$$
\Vert q\Vert_{H^s}\le {\cal O}(1)h^{-N_1+\frac{n}{2}}, \
\Vert q\Vert_{L^\infty }\le {\cal O}(1)h^{-N_1}, 
$$
such that if $P_{\delta}=P_{\delta _0}+\frac{1}{C}\tau
_0h^{2N_1+n}q=P+\delta Q$, $\delta =\frac{1}{C}h^{N_1+n}\tau_0$, $Q=h^{N_1}q$,
we have the following estimates on the singular values of 
$P_{\delta }-z$:
\begin{itemize}
\item If $\nu >N^{(0)}$, we have 
$t_\nu (P_{\delta }-z)\ge (1-\frac{h^{N_1+n}}{C})t_\nu (P-z)$.
\item If $N^{(k)}<\nu \le N^{(k-1)},$ $1\le k\le k_1$, then $
t_\nu (P_\delta -z)\ge (1-{\cal O}(h^{N_1+n}))\tau_0^{(k)}$.

\item Finally, for $\nu =N^{(k_1)}=1$, we have  $
t_1(P_\delta -z)\ge (1-{\cal O}(h^{N_1+n}))\tau_0^{(k_1+1)}$.
\end{itemize}
\end{prop}

\par As shown in \cite{Sj1} we have an equivalence between lower
bounds for the small singular values of $P_\delta -z$ in the above
proposition and for the singular values of $E_{-+}^\delta $ appearing
in the solution of the Grushin problem for $P_\delta -z$ (and that
is used in the proof of the proposition). We also have an equivalence
between lower bounds for the small singular values of $P_\delta -z$
and those of $P_{\delta ,z}$. For the latter operator we have a well
posed Grushin problem analogous to (\ref{en.4}) and an equivalence
between lower bounds for the small singular values of $P_{\delta ,z}$
and for the singular values of $E_{-+}^{\delta ,z}$, appearing in the
solution of the new Grushin problem. Using perturbative functional 
calculus we also have an asymptotic formula for $\ln \det {\cal
  P}_{\delta ,z}$, where 
$${\cal P}_{\delta ,z}=\left(\begin{array}{ccc}P_{\delta ,z}
    &R_-^z\\R_+^z &0 \end{array}\right)$$ is the matrix associated to
the new Grushin problem. As showed in \cite{HaSj} by means of
calculations from \cite{SjZw}, we have 
\ekv{en.10}
{
\det P_{\delta ,z}=\det {\cal P}_{\delta ,z}\det E_{-+}^{\delta ,z}.
}

\par The perturbative functional calculus gives a general upper bound
on $\ln \det P_{\delta ,z}$, and for the special admissible
perturbation in Proposition \ref{en1}, we have a lower bound on $\ln
|\det E_{-+}^{\delta ,z}|$ (using the lower bound on the singular
values of $E_{-+}^{\delta ,z}$ and the fact the modulus of the
determinant is equal to the product of the singular values). We get as
in \cite{Sj1}:
\begin{prop}\label{en2}
For the special admissible perturbation $P_\delta $ in
Proposition \ref{en1}, we have 
\eekv{en.11}
{
&&\ln |\det P_{\delta ,z}|\ge }
{&&\frac{1}{(2\pi h)^n}
\left( \iint \ln |p_z|dxd\xi -{\cal O}\left(
(h^{\kappa }+h^n\ln \frac{1}{h})(\ln \frac{1}{\tau_0}+
(\ln \frac{1}{h})^2)\right)
\right) .
}
\end{prop}

On the other hand, for more general operators of the form $P_\delta
=P_{\delta _0}+\tau_0h^{2N_1+n}q$ with $q$ admissible as in
(\ref{int.6.3}), (\ref{int.6.4}) we get as in \cite{Sj1} the upper
bound:
\ekv{en.12}
{
\ln |\det P_{\delta ,z}|\le \frac{1}{(2\pi h)^n}
\left( \iint \ln |p_z|dxd\xi +{\cal O}\left(h^{\kappa }\ln \frac{1}{h}\right)
\right) .
}

Section 8 of \cite{Sj1} (based on Jensen type arguments in the $\alpha
$-variables) now applies and shows that with probability close to 1,
we have 
\ekv{en.13}
{
\ln |\det P_{\delta ,z}|\approx \frac{1}{(2\pi h)^n}\iint \ln
|p_z|dxd\xi .
}

So, far $z$ was fixed and we now let it vary in a neighborhood of
$ \Gamma $, recalling that the eigenvalues $P_\delta $ in this region
coincide with the zeros of the holomorphic function $F_\delta (z)=\det
P_{\delta ,z}$. Assuming now that (\ref{int.6.2}) holds uniformly for
$z$ in a neighborhood of $\partial \Gamma $, we can then conlude the
proof as in Section 9 of \cite{Sj1}, by applying a general result
about the number of zeros of holomorphic functions with exponential
growth, from \cite{HaSj}. Recall that this result (applied to
$F_\delta (z)$) requires an upper bound on $\ln |F_\delta (z)|$ in a
fixed neighborhood of $\partial \Gamma $, in our case provided by
(\ref{en.12}), as well as a corresponding lower bound at finitely many
points along $\partial \Gamma $. The latter is provided by the lower
bound part of (\ref{en.13}) and holds with probability close to 1.

\end{document}